\documentclass [twoside,reqno,12pt] {amsart}

\usepackage{amsfonts}
\usepackage{amssymb}
\usepackage{a4}

\newtheorem{thm}{Theorem}[section]
\newtheorem{cor}[thm]{Corollary}

\newtheorem{prop}[thm]{Proposition}

\theoremstyle{definition}

\theoremstyle{remark}

\numberwithin{equation}{section}

\renewcommand{\Re}{\hbox{Re}\,}
\renewcommand{\Im}{\hbox{Im}\,}

\newcommand{\C}{\mathbb{C}}

\renewcommand{\div}{\operatorname{div}}

\newcommand{\R}{\mathbb{R}}

\def\tilde{\widetilde}
\def \bfo {\begin {eqnarray*} }
\def \efo {\end {eqnarray*} }
\def \ba {\begin {eqnarray*} }
\def \ea {\end {eqnarray*} }
\def \beq {\begin {eqnarray}}
\def \eeq {\end {eqnarray}}

\def \p {\partial}

\def\tilde{\widetilde}
\def \bfo {\begin {eqnarray*} }
\def \efo {\end {eqnarray*} }
\def \ba {\begin {eqnarray*} }
\def \ea {\end {eqnarray*} }
\def \beq {\begin {eqnarray}}
\def \eeq {\end {eqnarray}}

\def \p {\partial}

\begin{document}

 \title[Determining a first order perturbation ]{Determining a first order perturbation of the biharmonic operator by partial boundary measurements}

\author[Krupchyk]{Katsiaryna Krupchyk}

\address
        {K. Krupchyk, Department of Mathematics and Statistics \\
         University of Helsinki\\
         P.O. Box 68 \\
         FI-00014   Helsinki\\
         Finland}

\email{katya.krupchyk@helsinki.fi}

\author[Lassas]{Matti Lassas}

\address
        {M. Lassas, Department of Mathematics and Statistics \\
         University of Helsinki\\
         P.O. Box 68 \\
         FI-00014   Helsinki\\
         Finland}

\email{matti.lassas@helsinki.fi}

\author[Uhlmann]{Gunther Uhlmann}

\address
       {G. Uhlmann, Department of Mathematics\\
       University of Washington\\
       Seattle, WA  98195-4350\\
       and
       Department of Mathematics\\ 
       340 Rowland Hall\\
        University of California\\
        Irvine, CA 92697-3875\\
       USA}
\email{gunther@math.washington.edu}

\maketitle

\begin{abstract} 

We consider an operator  $\Delta^2 + A(x)\cdot D+q(x)$ with the Navier boundary conditions  on a bounded domain in  $\R^n$, $n\ge 3$. We show that a first order perturbation $A(x)\cdot D+q$  Êcan be determined uniquely by measuring the Dirichlet--to--Neumann map on possibly very small subsets of the boundary of the domain.  Notice that the corresponding result does not hold in general for a first order perturbation of the Laplacian. 

\end{abstract}

\section{Introduction and statement of results}

Let $\Omega\subset \R^n$, $n\ge 3$,  be a bounded  simply connected domain with $C^\infty$ connected boundary, and let us consider the following equation,
\[
\mathcal{L}_{A,q}(x,D)u=0\quad  \textrm{in}\quad \Omega,
\]
where
\[
\mathcal{L}_{A,q}(x,D)=\Delta^2+\sum_{j=1}^n A_j(x)D_j+q(x)=\Delta^2+A(x)\cdot D+q(x),
\]
$D=i^{-1}\nabla$, 
$A=(A_j)_{1\le j\le n}\in C^4(\overline{\Omega},\C^n)$,  and $q\in L^\infty(\Omega,\C)$. 
The operator $\mathcal{L}_{A,q}$, equipped with  the domain
\begin{equation}
\label{eq_domain_L}
\mathcal{D}(\mathcal{L}_{A,q})=\{ u\in H^{4}(\Omega): u|_{\p\Omega}=(\Delta u)|_{\p\Omega}=0\},
\end{equation}
is  an unbounded closed operator on $L^2(\Omega)$ with purely discrete spectrum, see \cite{Grubbbook2009}. 
The boundary conditions in \eqref{eq_domain_L} are called the Navier boundary conditions. 
Physically,  the operator $\mathcal{L}_{A,q}$ with the domain $\mathcal{D}(\mathcal{L}_{A,q})$  arises when considering the equilibrium configuration of an elastic plate which is hinged along the boundary, see \cite{GGS_book}.   Let us make the following assumption, 
\begin{itemize}
\item[\textbf{(A)}] $0$ is not an eigenvalue of $\mathcal{L}_{A,q}(x,D):\mathcal{D}(\mathcal{L}_{A,q})\to L^2(\Omega)$. 
\end{itemize}
Under the assumption (A), for any $(f_0,f_1)\in 
H^{7/2}(\p \Omega)\times H^{3/2}(\p \Omega)$,  
the boundary value problem 
\begin{equation}
\label{eq_Dirichlet_problem}
\begin{aligned}
\mathcal{L}_{A,q} u&=0\quad \textrm{in} \quad \Omega, \\
 u&=f_0\quad \textrm{on} \quad \p \Omega,\\
 \Delta u&=f_1\quad \textrm{on} \quad \p \Omega,
\end{aligned}
\end{equation}
has a unique solution $u\in H^{4}(\Omega)$. 
Let $\nu$ be the unit outer normal to the boundary $\p \Omega$.  We then  define the Dirichlet--to--Neumann map $\mathcal{N}_{A,q}$ by
\begin{align*}
\mathcal{N}_{A,q}: H^{7/2}(\p \Omega)\times H^{3/2}(\p \Omega) \to & H^{5/2}(\p \Omega)\times H^{1/2}(\p \Omega) , \\ 
\mathcal{N}_{A,q}(f_0,f_1)= &  (\p_\nu u|_{\p \Omega}, \p_\nu(\Delta u)|_{\p \Omega})
\end{align*}
where $u\in H^{4}(\Omega)$ is the solution to the problem  \eqref{eq_Dirichlet_problem}.  Let us also introduce the set of the Cauchy data $\mathcal{C}_{A,q}$ for the operator $\mathcal{L}_{A,q}$ defined as follows,
\[
\mathcal{C}_{A,q}=\{(u|_{\p \Omega}, (\Delta u)|_{\p \Omega},  \p_\nu u|_{\p \Omega}, \p_\nu(\Delta u)|_{\p \Omega} ): u\in H^{4}(\Omega), \mathcal{L}_{A,q}u=0\ \textrm{in }\Omega\}. 
\]
When the assumption (A) holds,  the set $\mathcal{C}_{A,q}$ is the graph of the Dirichlet--to--Neumann map $\mathcal{N}_{A,q}$.

It was shown in \cite{KrupLassasUhlmann} that the first order perturbation $A(x)\cdot D+ q(x)$ of the polyharmonic operator $(-\Delta)^m$, $m\ge 2$,  can be recovered from the knowledge of the set of the Cauchy data $\mathcal{C}_{A,q}$ on the boundary of $\Omega$.  Notice that for the corresponding problem for a first order perturbation of the Laplacian, this is no longer true, due to the gauge invariance of boundary measurements, see \cite{KKL,NakSunUlm_1995, Sun_1993}. 
 In this case, the first order perturbation can be recovered only modulo a gauge transformation,  \cite{NakSunUlm_1995,  Sun_1993}.

In this paper we are concerned with the inverse problem of determining  the first order perturbation $A(x)\cdot D+ q(x)$ of the biharmonic operator from the knowledge of the Dirichlet--to--Neumann map $\mathcal{N}_{A,q}$, given  only on a part of the boundary $\p \Omega$.

In many applications, say,  arising in geophysics,    performing measurements on the entire boundary could be either impossible or too cost consuming.  One is therefore naturally led to an inverse boundary value problem with partial measurements. Substantial progress has been made recently on the partial data problems in the context of electrical impedance tomography as well as the 
Schr\"odinger equation.  Specifically,  in the paper \cite{BukhUhl_2002} it was shown that an unknown conductivity is determined uniquely by performing voltage--to--current measurements on, roughly speaking, a half of the boundary. The main technical tool was a boundary Carleman estimate with a linear weight.  The Carleman estimates approach to the partial data problem was very much advanced in the work 
\cite{KenSjUhl2007}.  Here, rather than working with linear weights, a broader class of limiting Carleman weights was introduced and employed.  The work \cite{KenSjUhl2007} contains a global uniqueness result for the conductivity equation,  assuming that the voltage--to--current map is measured on a possibly very small subset of the boundary, with  the precise shape depending
on the geometry of the boundary.  The limiting Carleman weights approach of  \cite{KenSjUhl2007} has led to subsequent important developments in the partial data problem for the magnetic Schr\"odinger operator  \cite{DKSU_2007}, \cite{KnuSalo_2007} and  the Dirac system  
\cite{Salo_Tzou_2010}.   To the best of our knowledge, applications of this approach to partial data problems for  other important equations and systems of mathematical physics have not yet been explored.  The purpose of this paper is to apply the techniques of   Carleman estimates to the partial data problem for the perturbed biharmonic operator. 

We should also mention 
another approach to the partial data problems for the conductivity equation, which  is due to \cite{Isakov_2007}, and which is based on reflection arguments.  In this approach, the subset of the boundary, where the measurements are performed is such that the inaccessible part of the boundary is a subset of a hyperplane or a sphere.  The work \cite{COS_2009} gives a partial data result
analogous to \cite{Isakov_2007} for the Maxwell equations.

Let us now proceed to describe the precise assumptions and results. Let $x_0\in \R^n\setminus\overline{\textrm{ch}(\Omega)}$,  where  $\textrm{ch}(\Omega)$ is the convex hull of $\Omega$. Following \cite{KenSjUhl2007}, we define the front face of $\p \Omega$ with respect to $x_0$ by
\begin{equation}
\label{eq_intr_F(x_0)}
F(x_0)=\{x\in \p \Omega: (x-x_0)\cdot \nu(x)\le 0\},
\end{equation}
and let $\tilde F$ be an open neighborhood of $F(x_0)$ in $\p \Omega$. 
The main result of this paper is as follows. 

\begin{thm}
\label{thm_main} Let $\Omega\subset \R^n$, $n\ge 3$,  be  a bounded  simply connected domain with $C^\infty$ connected boundary,  and let $A^{(1)}, A^{(2)}\in  C^4(\overline{\Omega},\C^n)$ and $q^{(1)},q^{(2)}\in L^\infty(\Omega,\C)$, be such that the assumption (A) is satisfied for both operators. 
If 
\[
\mathcal{N}_{A^{(1)},q^{(1)}}(f_0,f_1)|_{\tilde F}=\mathcal{N}_{A^{(2)},q^{(2)}}(f_0,f_1)|_{\tilde F}\quad \textrm{for all  } (f_0,f_1)\in H^{7/2}(\p \Omega)\times H^{3/2}(\p \Omega), 
\]
 then $A^{(1)}=A^{(2)}$ and $q^{(1)}=q^{(2)}$ in $\Omega$.  
\end{thm}

Following \cite{KenSjUhl2007}, we say that an open set $\Omega\subset \R^n$ with smooth boundary is strongly star shaped with respect to a point $x_1\in \p \Omega$, if every line through $x_1$ which is not contained in the tangent hyperplane cuts the boundary $\p \Omega$ at precisely two distinct points $x_1$ and $x_2$, with transversal intersection at $x_2$.  We have the following corollary of Theorem \ref{thm_main}.

\begin{cor}
\label{cor_1}
Let $\Omega\subset \R^n$, $n\ge 3$,  be  a bounded  simply connected domain with $C^\infty$ connected boundary, and let $x_1\in\p \Omega$ be such that the tangent hyperplane of $\p \Omega$ at $x_1$ only intersects $\p \Omega$ at $x_1$ and $\Omega$ is strongly star shaped with respect to $x_1$. 
Furthermore,  let $A^{(1)}, A^{(2)}\in  C^4(\overline{\Omega},\C^n)$ and $q^{(1)},q^{(2)}\in L^\infty(\Omega,\C)$, be such that the assumption (A) is satisfied for both operators.  
If for a neighborhood $\tilde F$ of $x_1$ in $\p \Omega$, we have
\[
\mathcal{N}_{A^{(1)},q^{(1)}}(f_0,f_1)|_{\tilde F}=\mathcal{N}_{A^{(2)},q^{(2)}}(f_0,f_1)|_{\tilde F}\quad \textrm{for all  } (f_0,f_1)\in H^{7/2}(\p \Omega)\times H^{3/2}(\p \Omega), 
\]
 then $A^{(1)}=A^{(2)}$ and $q^{(1)}=q^{(2)}$ in $\Omega$.  
\end{cor} 

Notice that if $\Omega$ is strictly convex, then  the assumptions on $\Omega$ of Corollary \ref{cor_1} are satisfied for any $x_1\in\p \Omega$, and therefore,  measuring the Dirichlet--to--Neumann map on an arbitrarily small open subset of the boundary determines the first order perturbation uniquely.

Let $\tilde F$ be an open neighborhood of the front face of $\p \Omega$ with respect to $x_0$, defined in 
\eqref{eq_intr_F(x_0)}. 
Associated to $\tilde F$, consider the set of the Cauchy data for the first order perturbation of the biharmonic operator, which is based on the Dirichlet boundary conditions   $(u|_{\p \Omega}, \p_\nu u|_{\p \Omega})$, 
\[
\tilde{\mathcal{C}}_{A^{(j)},q^{(j)}}^{\tilde F}=\{(u|_{\p \Omega}, \p_\nu u|_{\p \Omega},  \p_\nu^2 u|_{\p \Omega}  , \p_\nu^3 u|_{\tilde F} ): u\in H^{4}(\Omega), \mathcal{L}_{A^{(j)},q^{(j)}}u=0\ \textrm{in }\Omega\},
\]
$j=1,2$.  Notice that the Dirichlet boundary conditions correspond to the clamped plate equation.   
We have the following partial data result.  

\begin{cor}
\label{cor_cor}
Let $\Omega\subset \R^n$, $n\ge 3$,  be  a bounded  simply connected domain with $C^\infty$ connected boundary,  and let $A^{(1)}, A^{(2)}\in  C^4(\overline{\Omega},\C^n)$ and $q^{(1)},q^{(2)}\in L^\infty(\Omega,\C)$. 
If $\tilde{\mathcal{C}}_{A^{(1)},q^{(1)}}^{\tilde F}=\tilde{\mathcal{C}}_{A^{(2)},q^{(2)}}^{\tilde F}$, 
 then $A^{(1)}=A^{(2)}$ and $q^{(1)}=q^{(2)}$ in $\Omega$.  
\end{cor}

Corollary \ref{cor_cor} follows from Theorem \ref{thm_main} and the explicit description for the Laplacian in the boundary normal coordinates, see \cite{LeeUhl89}.

Finally, let us mention that the study of inverse boundary value problems has a long and distinguished tradition, in particular, in the context of electrical impedance tomography, see    
\cite{AP, ALP,  N2, S} for the 
two dimensional case,
 and \cite{Cal, GLU1, N1, PPU, SU} for the case of higher dimensions,
 as well as in  inverse boundary value problems and 
inverse scattering problems for the Schr\"odinger equation \cite{Bukh_2008, EskRal_1995, GLU1, Ima_Uhl_Yam_2010,  NakSunUlm_1995,  No,  Salo_2006, Sun_1993}, 
and in  elliptic inverse problems on Riemannian manifolds,
\cite{GT1, GT2, KKL, LTU, LU, LeeUhl89}. 
For sufficiently non-regular coefficients, the inverse problems are
not uniquely solvable, see \cite{GLU1, GLU2}, even when the measurements
are performed on the whole boundary. These 
counterexamples are  closely related to the so-called invisibility cloaking, see e.g.\ \cite{GKLU1, GKLUbull, KSVW, Le, PSS1}.

The paper is organized as follows.  In Section 2 we construct complex geometric optics solutions for the perturbed biharmonic operator, using the 
 methods of Carleman estimates with limiting Carleman weights.  Section 3 is devoted to Carleman estimates with boundary terms for the perturbed biharmonic operator. These estimates are crucial  when  estimating away the boundary terms  in the derivation of the main integral identity, which is carried out in Section 4.  The final Section 5 is concerned with the determination of the first order perturbation, relying upon the main integral identity.  We notice that the unique identifiability of the vector field part of the perturbation becomes possible thanks to special properties of the amplitudes in the complex geometric optics solutions.

\section{Construction of complex geometric optics solutions} 

\label{sec_CGO}

Let $\Omega\subset\R^n$, $n\ge 3$, be a bounded domain with $C^\infty$-boundary.  Following \cite{DKSU_2007, KenSjUhl2007}, we shall use the method of Carleman estimates to construct complex geometric optics solutions for the equation $\mathcal{L}_{A,q}u=0$ in $\Omega$,  with $A\in  C^4(\overline{\Omega},\C^n)$ and $q\in L^\infty(\Omega,\C)$. 

First we shall derive a Carleman estimate for the semiclassical biharmonic operator $(-h^2\Delta)^2$, where $h>0$ is a small parameter,  by iterating the corresponding Carleman estimate for the semiclassical Laplacian $-h^2\Delta$, which we now proceed to recall following  \cite{KenSjUhl2007, Salo_Tzou_2009}. 
Let $\tilde \Omega$ be an open set in $\R^n$ such that $ \Omega\subset\subset\tilde \Omega$ and 
$\varphi\in C^\infty(\tilde \Omega,\R)$.  Consider the conjugated operator 
\[
P_\varphi=e^{\frac{\varphi}{h}}(-h^2\Delta) e^{-\frac{\varphi}{h}}
\]
and its semiclassical principal symbol
\begin{equation}
\label{eq_p_varphi}
p_\varphi(x,\xi)=\xi^2+2i\nabla \varphi\cdot \xi-|\nabla \varphi|^2, \quad x\in \tilde {\Omega},\quad  \xi\in \R^n. 
\end{equation}
Following \cite{KenSjUhl2007}, we say that $\varphi$ is a limiting Carleman weight for $-h^2\Delta$ in $\tilde \Omega$, if $\nabla \varphi\ne 0$ in $\tilde \Omega$ and the Poisson bracket of $\Re p_\varphi$ and $\Im p_\varphi$ satisfies, 
\[
\{\Re p_\varphi,\Im p_\varphi\}(x,\xi)=0 \quad \textrm{when}\quad p_\varphi(x,\xi)=0, \quad (x,\xi)\in \tilde{\Omega}\times \R^n. 
\]
Examples are  linear weights $\varphi(x)=\alpha\cdot x$, $\alpha\in \R^n$, $|\alpha|=1$, and logarithmic weights $\varphi(x)=\log|x-x_0|$,  with $x_0\not\in \tilde \Omega$.  In this paper we shall use the logarithmic weights. 

In what follows we shall equip  the standard Sobolev space $H^s(\R^n)$, $s\in\R$, with the semiclassical norm $\|u\|_{H^s_{\textrm{scl}}}=\|\langle hD \rangle^s u\|_{L^2}$. Here $\langle\xi  \rangle=(1+|\xi|^2)^{1/2}$. 
We shall need the following result, obtained in \cite{Salo_Tzou_2009}, generalizing the Carleman estimate with a gain of one derivative, proven in  \cite{KenSjUhl2007}.

\begin{prop}
Let $\varphi$ be a limiting Carleman weight for the semiclassical Laplacian on $\tilde \Omega$. Then the Carleman estimate 
\begin{equation}
\label{eq_Carleman_lap}
\|e^{\frac{\varphi}{h}}(-h^2\Delta)e^{-\frac{\varphi}{h}}u\|_{H^s_{\emph{scl}}}\ge \frac{h}{C_{s,\Omega}}\|u\|_{H^{s+2}_{\emph{scl}}},\quad C_{s,\Omega}>0,
\end{equation}
holds for all $u\in C^\infty_0(\Omega)$, $s\in\R$ and  all $h>0$ small enough.  
\end{prop}

Iterating the Carleman estimate \eqref{eq_Carleman_lap} two times, we get the following Carleman estimate for the biharmonic operator,
\begin{equation}
\label{eq_Carleman_poly_1}
\|e^{\frac{\varphi}{h}}(h^2\Delta)^2 e^{-\frac{\varphi}{h}}u\|_{H^s_{\textrm{scl}}}\ge \frac{h^2}{C_{s,\Omega}}\|u\|_{H^{s+4}_{\textrm{scl}}},
\end{equation}
for all $u\in C^\infty_0(\Omega)$, $s\in\R$ and $h>0$ small. 

Let $A\in W^{1,\infty}(\Omega,\C^n)$ and $q\in L^\infty(\Omega,\C)$. Then to add the perturbation $h^{4}q$ to the estimate \eqref{eq_Carleman_poly_1}, we assume that 
$-4\le s\le 0$ and use that 
\[
\|qu\|_{H^s_{\textrm{scl}}}\le \|qu\|_{L^2}\le \|q\|_{L^\infty}\|u\|_{L^2}\le  \|q\|_{L^\infty}\|u\|_{H^{s+4}_{\textrm{scl}}}.
\] 
To add the perturbation 
\[
h^{3}e^{\frac{\varphi}{h}}(A\cdot hD) e^{-\frac{\varphi}{h}}=h^{3}(A\cdot hD+iA\cdot\nabla \varphi) 
\]
 to the estimate \eqref{eq_Carleman_poly_1}, assuming that 
$-4\le s\le 0$, we need the following estimates 
\[
\|(A\cdot\nabla \varphi)u\|_{H^s_{\textrm{scl}}}\le \|A\cdot\nabla \varphi\|_{L^\infty}\|u\|_{H^{s+4}_{\textrm{scl}}},
\]
\begin{align*}
\|A\cdot hD u\|_{H^s_{\textrm{scl}}}&\le \sum_{j=1}^n\|hD_j(A_j u)\|_{H^s_{\textrm{scl}}}+\mathcal{O}(h)\|(\div A)u\|_{H^s_{\textrm{scl}}}\\
&\le \mathcal{O}(1)\sum_{j=1}^n\|A_j u\|_{H^{s+1}_{\textrm{scl}}}+\mathcal{O}(h)\|u\|_{H^{s+4}_{\textrm{scl}}}\le \mathcal{O}(1)\|u\|_{H^{s+4}_{\textrm{scl}}}.
\end{align*}
When obtaining the last inequality, we notice that the operator, given by multiplication by $A_j$, maps $H^{s+4}_{\textrm{scl}}\to H^{s+1}_{\textrm{scl}}$. To see this by complex interpolation it suffices to consider the cases $s=0$ and $s=-4$. 

Let 
\[
\mathcal{L}_\varphi=e^{\frac{\varphi}{h}} h^{4}\mathcal{L}_{A,q} e^{-\frac{\varphi}{h}}.
\]
Thus, we obtain the following Carleman estimate for a first order perturbation of the biharmonic operator.

\begin{prop} 

Let $A\in W^{1,\infty}(\Omega,\C^n)$,  $q\in L^\infty(\Omega,\C)$, and $\varphi$ be a limiting Carleman weight for the semiclassical Laplacian on $\tilde \Omega$.  If $-4\le s\le 0$, then for $h>0$ small enough,  one 
has 
\begin{equation}
\label{eq_Carleman_poly_perturbation}
\|\mathcal{L}_\varphi u\|_{H^s_{\emph{scl}}}\ge \frac{h^2}{C_{s,\Omega,A,q}}\|u\|_{H^{s+4}_{\emph{scl}}},
\end{equation}
for all $u\in C^\infty_0(\Omega)$. 
\end{prop}

The formal $L^2$-adjoint of $\mathcal{L}_\varphi$ is given by
\[
\mathcal{L}_\varphi^* =e^{-\frac{\varphi}{h}} (h^{4}\mathcal{L}_{\bar{A},i^{-1}\nabla\cdot \bar{A}+\bar q}) e^{\frac{\varphi}{h}}. 
\]
Notice that  if $\varphi$ is a limiting Carleman weight, then so is $-\varphi$. This implies that the Carleman estimate \eqref{eq_Carleman_poly_perturbation} holds also for the formal adjoint $\mathcal{L}_\varphi^*$.

To construct complex geometric optics solution we need the following solvability result, similar to \cite{DKSU_2007}.  The proof is essentially well-known, and is included here for the convenience of the reader.

\begin{prop}
\label{prop_Hahn-Banach}
Let $A\in W^{1,\infty}(\Omega,\C^n)$,  $q\in L^\infty(\Omega,\C)$, and let $\varphi$ be a limiting Carleman weight for the semiclassical  Laplacian on $\tilde \Omega$. 
If $h>0$ is small enough, then for any $v\in L^2(\Omega)$, there is a solution $u\in H^4(\Omega)$ of the equation
\[
\mathcal{L}_\varphi u=v \quad \text{in} \quad \Omega,
\]
which satisfies 
\[
\|u\|_{H^4_{\emph{scl}}}\le \frac{C}{h^2} \|v\|_{L^2}.
\] 

\end{prop}

\begin{proof}
Consider the following complex linear functional 
\[
L: \mathcal{L}_\varphi^* C^\infty_0(\Omega)\to \C, \quad \mathcal{L}_\varphi^* w\mapsto (w,v)_{L^2}. 
\]
By the Carleman estimate \eqref{eq_Carleman_poly_perturbation} for the formal adjoint $\mathcal{L}_\varphi^*$, the map $L$ is well-defined. 

Let $w\in  C^\infty_0(\Omega)$. We have
\[
|L(\mathcal{L}_\varphi^* w)|=|(w,v)_{L^2}|\le \|w\|_{L^2} \|v\|_{L^2}\le \frac{C}{h^2}\|\mathcal{L}_\varphi^* w\|_{H^{-4}_{\textrm{scl}}} \|v\|_{L^2},
\]
showing that $L$ is bounded in the $H^{-4}$-norm. Thus, by the Hahn-Banach theorem, we may extend $L$ to a linear continuous functional $\tilde L$ on $H^{-4}(\R^n)$ without increasing the norm. By the Riesz representation theorem, there exists $u\in H^4(\R^n)$ such that for all $w\in H^{-4}(\R^n)$, 
\[
\tilde L(w)=(w,u)_{(H^{-4},H^4)}, \quad \textrm{and}\quad \|u\|_{H^{4}_{\textrm{scl}}}\le \frac{C}{h^2}\|v\|_{L^2}. 
\]
Here $(\cdot,\cdot)_{(H^{-4},H^4)}$ stands for the usual $L^2$-duality. It follows that $\mathcal{L}_\varphi u=v$ in $\Omega$. This completes the proof. 
\end{proof}

Our next goal is to construct complex geometric optics solutions of  the equation 
\begin{equation}
\label{eq_new_main}
\mathcal{L}_{A,q}u=0\quad  \textrm{in} \quad \Omega,
\end{equation}
with $A\in C^4(\overline{\Omega},\C^4)$ and $q\in L^\infty(\Omega,\C)$,  i.e. solutions of the following form,
\begin{equation}
\label{eq_new_CGO_1}
u(x;h)=e^{\frac{\varphi+i\psi}{h}} (a_0(x)+ha_1(x)+ r(x; h)).
\end{equation}
Here $\varphi\in C^\infty(\tilde \Omega,\R)$ is a limiting Carleman weight for the semiclassical Laplacian on $\tilde \Omega$, $\psi\in C^\infty(\tilde \Omega,\R )$ is a solution to the eikonal equation
$p_{\varphi}(x,\nabla \psi)=0\quad \textrm{in}\quad  \tilde \Omega$,
where $p_\varphi$  is given by \eqref{eq_p_varphi}, i.e.
\begin{equation}
\label{eq_eikonal}
|\nabla \psi|^2=|\nabla \varphi|^2,\quad \nabla \varphi\cdot \nabla \psi=0, \quad \textrm{in}\quad \tilde \Omega,  
\end{equation}
  the amplitudes $a_0\in C^\infty(\overline{\Omega})$ and $a_1\in C^4(\overline{\Omega})$  are solutions of the first and second transport equations, and $r$ is a correction term, satisfying $\|r\|_{H^4_{\textrm{scl}}(\Omega)}= \mathcal{O}(h^2)$.

Following \cite{DKSU_2007, KenSjUhl2007}, we fix a point $x_0\in \R^n\setminus{\overline{\textrm{ch}(\Omega)}}$ and let the limiting Carleman weight  be 
\begin{equation}
\label{eq_varphi}
\varphi(x)=\frac{1}{2}\log|x-x_0|^2,
\end{equation}
and 
\begin{equation}
\label{eq_psi}
\psi(x)=\frac{\pi}{2}-\arctan\frac{\omega\cdot (x-x_0)}{\sqrt{(x-x_0)^2-(\omega\cdot(x-x_0))^2}}=\textrm{dist}_{\mathbb{S}^{n-1}}\bigg(\frac{x-x_0}{|x-x_0|},\omega\bigg),
\end{equation}
where $\omega\in \mathbb{S}^{n-1}$ is chosen so that $\psi$ is smooth near $\overline{\Omega}$. Thus, given  $\varphi$, the function  $\psi$ satisfies the eikonal equation \eqref{eq_eikonal} near $\overline{\Omega}$.

Consider the conjugated operator, 
\begin{equation}
\label{eq_indep_h}
e^{\frac{-(\varphi+i\psi)}{h}} h^{4}\mathcal{L}_{A,q} e^{\frac{\varphi+i\psi}{h}}=
(h^2\Delta +2hT )^2
+ h^{3}A\cdot hD+h^{3}A\cdot (D\varphi+iD\psi)  +h^{4}q,
\end{equation}
where 
\begin{equation}
\label{eq_operator_T}
T=(\nabla\varphi+i\nabla \psi)\cdot\nabla +\frac{1}{2}(\Delta \varphi + i \Delta \psi).
\end{equation}

Substituting 
\eqref{eq_new_CGO_1} into \eqref{eq_new_main} and collecting powers of $h$, we get
the first transport equation, 
\begin{equation}
\label{eq_transport_1}
T^2 a_0=0\quad \textrm{in}\quad \Omega,
\end{equation}
the second transport equation
\begin{equation}
\label{eq_transport_2}
T^2 a_1=-\frac{1}{2}(\Delta\circ T+T\circ \Delta)a_0-\frac{1}{4}A\cdot (D\varphi+iD\psi)a_0 \quad \textrm{in}\quad \Omega,
\end{equation}
and 
\begin{equation}
\label{eq_O(h)}
\begin{aligned}
e^{\frac{-(\varphi+i\psi)}{h}} h^{4}\mathcal{L}_{A,q} (e^{\frac{\varphi+i\psi}{h}} r)=&-e^{\frac{-(\varphi+i\psi)}{h}} h^{4}\mathcal{L}_{A,q} (e^{\frac{\varphi+i\psi}{h}} (a_0+ha_1))\\
=&-h^4\Delta^2(a_0+ha_1)-2h^4(\Delta\circ T+T\circ \Delta)a_1-h^4A\cdot Da_0\\
&-h^5A\cdot Da_1-h^4A\cdot (D\varphi+iD\psi)a_1-h^4q(a_0+ha_1).
\end{aligned}
\end{equation}

Let us now discuss the solvability of the equations \eqref{eq_transport_1}, \eqref{eq_transport_2} and \eqref{eq_O(h)}. 
  To this end  we follow \cite{DKSU_2007} and choose coordinates in $\R^n$ so that $x_0=0$ and $\overline{\tilde \Omega}\subset \{x_n>0\}$.  We set $\omega=e_1=(1,0_{\R^{n-1}})$, and introduce also the  cylindrical coordinates $(x_1,r\theta)$ on $\R^n$ with $r>0$ and $\theta\in \mathbb{S}^{n-2}$.  Consider the change of coordinates $x\mapsto (z,\theta)$, where $z=x_1+i r$ is a complex variable. Notice that $\Im z>0$ near $\overline{\Omega}$.  Then in these coordinates,  we have 
\[
\varphi=\log|z|=\Re\log z, \quad \psi=\frac{\pi}{2}-\arctan\frac{\Re z}{\Im z}=\Im \log z,
\]
when $\Im z>0$. 
Hence,
\[
\varphi+i\psi=\log z. 
\]
Moreover, we have 
\begin{equation}
\label{eq_nabla_cyl}
\nabla(\varphi+i\psi)=\frac{1}{z}(e_1+ie_r),
\end{equation}
where $e_r=(0,\theta)$, $\theta\in \mathbb{S}^{n-2}$, and 
\begin{equation}
\label{eq_nabla_cyl_2}
\nabla(\varphi+i\psi)\cdot \nabla=\frac{2}{z}\p_{\bar z}, \quad \Delta(\varphi+i\psi)=-\frac{2(n-2)}{z(z-\bar z)}. 
\end{equation}
In the cylindrical coordinates, the operator $T$, defined in \eqref{eq_operator_T}, has the form,  
\[
T=\frac{2}{z}\bigg(\p_{\bar z}-\frac{(n-2)}{2(z-\bar z)}\bigg). 
\]
Thus, it follows that 
\eqref{eq_transport_1}   has the form,
\[
\bigg(\p_{\bar z}-\frac{(n-2)}{2(z-\bar z)}\bigg)^2a_0=0\quad \textrm{in} \quad \Omega.
\]
In particular, one can take $a_0\in C^\infty(\overline{\Omega})$  satisfying $\bigg(\p_{\bar z}-\frac{(n-2)}{2(z-\bar z)}\bigg)a_0=0$. The general solution of the latter equation is given by
$ a_0=(z-\bar z)^{(2-n)/2}g_0$ with $g_0\in C^\infty(\overline{\Omega})$ satisfying $\p_{\bar z}g_0=0$. 

In the cylindrical coordinates,   the second transport equation \eqref{eq_transport_2} has the form,
\begin{equation}
\label{eq_transport_2_cyl}
\bigg(\p_{\bar z}-\frac{(n-2)}{2(z-\bar z)}\bigg)^2a_1=f\quad \textrm{in} \quad \Omega,
\end{equation}
where $f$ is given.  Notice that  $a_1$ will have in general the same regularity as $f$, which is the same as the regularity of $A$.    
It follows from \eqref{eq_O(h)} that we need four derivatives of $a_1$, which explains our regularity assumptions on $A$, i.e. $A\in C^4(\overline{\Omega},\C^n)$.  

In order to solve \eqref{eq_transport_2_cyl}, given $f\in C^4(\overline{\Omega})$,  one can find $v\in C^4(\overline{\Omega})$, which satisfies 
\begin{equation}
\label{eq_transport_sol_2}
\bigg(\p_{\bar z}-\frac{(n-2)}{2(z-\bar z)}\bigg)v=f \quad \textrm{in} \quad \Omega,
\end{equation}
and then solve 
\[
\bigg(\p_{\bar z}-\frac{(n-2)}{2(z-\bar z)}\bigg)a_1=v \quad \textrm{in} \quad \Omega.
\]
We look for a solution of \eqref{eq_transport_sol_2} in the form $v=e^{g}v_0$ with $g\in C^\infty(\overline{\Omega})$ satisfying $\p_{\bar z}g=\frac{n-2}{2(z-\bar z)}$. 
Thus, $v_0\in C^4(\overline{\Omega})$ can be obtained by solving $\p_{\bar z} v_0=e^{-g}f$, applying the Cauchy transform, i.e.
\[
v_0(z,\theta)=\frac{1}{\pi}\int_{\C}\frac{\chi(z-\zeta,\theta)e^{-g(z-\zeta,\theta)}f(z-\zeta,\theta)}{\zeta}d\Re \zeta d\Im \zeta,
\]
where $\chi\in C^\infty_0(\R^n)$ is such that $\chi=1$ near $\overline{\Omega}$. 
  Hence, the second transport equation \eqref{eq_transport_2} is solvable globally near $\overline{\Omega}$ with a solution $a_1\in C^4(\overline{\Omega})$.

Having chosen the amplitudes $a_0\in C^\infty(\overline{\Omega})$ and $a_1\in C^4(\overline{\Omega})$, we obtain from \eqref{eq_O(h)}
that 
\[
e^{\frac{- \varphi}{h}} h^{4} \mathcal{L}_{A,q} e^{\frac{\varphi}{h}} (e^{\frac{i\psi}{h}} r)=\mathcal{O}(h^4), 
\]
in $L^2(\Omega)$.  Thanks to 
Proposition \ref{prop_Hahn-Banach}, for $h>0$ small enough,  there exists a solution $r\in H^4(\Omega)$ of \eqref{eq_O(h)} such that $\|r\|_{H^4_{\textrm{scl}}}=\mathcal{O}(h^2)$. 

Summing up, we have the following result.

\begin{prop}

\label{prop_complex_geom_optics}

Let $A\in C^4(\overline{\Omega},\C^n)$,  $q\in L^\infty(\Omega,\C)$. Then for all $h>0$ small enough,  there exist solutions $u(x; h)\in H^4(\Omega)$ to the equation $\mathcal{L}_{A,q}u=0$ in $\Omega$, of the form
\[
u(x;h)=e^{\frac{\varphi+i\psi}{h}} (a_0(x)+ha_1(x)+ r(x; h)),
\]  
where  $\varphi\in C^\infty(\tilde \Omega,\R)$ is a limiting Carleman weight for the semiclassical Laplacian on $\tilde \Omega$, chosen as in  \eqref{eq_varphi}, $\psi\in C^\infty(\tilde \Omega,\R )$ is given by  \eqref{eq_psi}, 
$a_0\in C^\infty(\overline{\Omega})$ and $a_1\in C^4(\overline{\Omega})$ are solutions of the first and second transport equations   \eqref{eq_transport_1} and  \eqref{eq_transport_2}, respectively,  and $r$ is a correction term, satisfying $\|r\|_{H^4_{\emph{\textrm{scl}}}(\Omega)}= \mathcal{O}(h^2)$.  

\end{prop}

\section{Boundary Carleman estimates} Following \cite{DKSU_2007, KenSjUhl2007}, in order to prove that some boundary integrals tend to zero as $h\to 0$, we shall use Carleman estimates, involving the boundary terms. 

Let $\Omega\subset\R^n$, $n\ge 3$, be a bounded domain with $C^\infty$-boundary, and $\tilde \Omega\subset \R^n$ be an open set such that $\Omega\subset\subset \tilde\Omega$, and $\varphi\in C^\infty(\tilde \Omega,\R)$ be a limiting Carleman weight for the semiclassical Laplacian. We define 
\[
\p \Omega_\pm=\{x\in\p \Omega:\pm \p_{\nu}\varphi (x)\ge 0\}. 
\] 
We shall need the following result from \cite{KenSjUhl2007}.

\begin{prop}
Let $\varphi$ be a limiting Carleman weight for the semiclassical Laplacian on $\tilde \Omega$. Then there exists $C>0$ such that  the following Carleman estimate 
\begin{equation}
\label{eq_Carleman_lap_boundary}
\begin{aligned}
\|e^{\frac{-\varphi}{h}}(-h^2\Delta)u\|_{L^2}+ h^{3/2}\|\sqrt{-\p_\nu\varphi}\, e^{-\frac{\varphi}{h}} \p_\nu u\|_{L^2(\p \Omega_-)}\\
\ge \frac{1}{C}(h\|e^{\frac{-\varphi}{h}}u\|_{H^{1}_{\emph{scl}}}+h^{3/2}\|\sqrt{\p_\nu \varphi}\, e^{-\frac{\varphi}{h}} \p_\nu u\|_{L^2(\p \Omega_+)}),
\end{aligned}
\end{equation}
holds for all $u\in H^2(\Omega)$,  $u|_{\p \Omega}=0$, and  all $h>0$ small enough.  
\end{prop}

Iterating \eqref{eq_Carleman_lap_boundary} two times and adding a first order perturbation, we get the following boundary Carleman estimate for the first order perturbation of the biharmonic operator.  
 
\begin{prop} 

Let $A\in W^{1,\infty}(\Omega,\C^n)$,  $q\in L^\infty(\Omega,\C)$, and $\varphi$ be a limiting Carleman weight for the semiclassical Laplacian on $\tilde \Omega$.  Then the following estimate
\begin{equation}
\label{eq_boundary_Carleman}
\begin{aligned}
\|e^{\frac{-\varphi}{h}}&(h^4 \mathcal{L}_{A,q})u\|_{L^2}+ h^{3/2}\|\sqrt{-\p_\nu\varphi}e^{-\frac{\varphi}{h}}\p_\nu(-h^2\Delta u)\|_{L^2(\p \Omega_-)}\\
&+ h^{5/2}\|\sqrt{-\p_\nu\varphi}e^{-\frac{\varphi}{h}}\p_\nu u\|_{L^2(\p \Omega_-)}\\
&\ge \frac{1}{C}(h^2\|e^{\frac{-\varphi}{h}}u\|_{H^{1}_{\emph{scl}}}+h^{3/2}\|\sqrt{\p_\nu \varphi}e^{-\frac{\varphi}{h}}\p_\nu (-h^2\Delta u)\|_{L^2(\p \Omega_+)}\\
&+h^{5/2}\|\sqrt{\p_\nu \varphi}e^{-\frac{\varphi}{h}}\p_\nu u\|_{L^2(\p \Omega_+)}),
\end{aligned}
\end{equation}
holds, 
for all $u\in H^4(\Omega)$,  $u|_{\p \Omega}=(\Delta u)|_{\p \Omega}=0$,  and  all $h>0$ small enough.  

\end{prop}

Here we use that
\begin{align*}
\|e^{-\frac{\varphi}{h}}h^4A\cdot Du\|_{L^2}\le  h^3 \sum_{j=1}^n\bigg(\|h D_j (A_je^{-\frac{\varphi}{h}} u)\|_{L^2}+\| (hD_j(A_je^{-\frac{\varphi}{h}}))u\|_{L^2}\bigg)\\
\le  h^3 \sum_{j=1}^n\bigg(2\|(hD_jA_j)e^{-\frac{\varphi}{h}} u\|_{L^2}+\| A_jhD_j(e^{-\frac{\varphi}{h}}u)\|_{L^2}+ \| A_j(D_j\varphi)e^{-\frac{\varphi}{h}}u\|_{L^2}\bigg)\\
\le \mathcal{O}(h^3)\|e^{-\frac{\varphi}{h}} u\|_{H^1_{\textrm{scl}}}.
\end{align*}

Notice that when $\varphi$ is given by \eqref{eq_varphi}, we have $\p_\nu \varphi(x)=\frac{(x-x_0)\cdot\nu(x)}{|x-x_0|^2}$ and therefore, $\p \Omega_-=F(x_0)$. 

\section{Integral identity, needed to recover the first order perturbation}

We shall need the following Green's formula, see \cite{Grubbbook2009},
\begin{equation}
\label{eq_Green_form}
\begin{aligned}
&\int_{\Omega} (\mathcal{L}_{A,q}u)\overline{v} dx-\int_{\Omega} u\overline{\mathcal{L}_{A,q}^*v} dx=-i\int_{\p \Omega} \nu(x)\cdot u A \overline{v} dS -\int_{\p \Omega} \p_\nu (-\Delta u) \overline{v}dS\\
 &  +\int_{\p \Omega} (-\Delta u)\overline{\p_\nu v }dS - \int_{\p \Omega} \p_{\nu} u\overline{( -\Delta v)} dS
  +\int_{\p \Omega} u\overline{(\p_\nu (-\Delta v))}dS,
\end{aligned}
\end{equation}
valid for all $u,v\in H^4(\Omega)$. Here  $\mathcal{L}_{A,q}^*=\mathcal{L}_{\bar{A},i^{-1}\nabla \cdot \bar{A}+\bar q}$ is the adjoint of $\mathcal{L}_{A,q}$, $\nu$ is the unit outer normal to the boundary $\p \Omega$,  and $dS$ is the surface measure on $\p \Omega$. 

Let $(f_0,f_1)\in H^{7/2}(\p \Omega)\times H^{3/2}(\p \Omega)$ and $u_j\in H^4(\Omega)$  be such that 
\[
\mathcal{L}_{A^{(j)},q^{(j)}}u_j=0 \  \textrm{in} \  \Omega,\ j=1,2,\quad u_1|_{\p \Omega}=u_2|_{\p \Omega}=f_0,\ (\Delta u_1)|_{\p \Omega}=(\Delta u_2)|_{\p \Omega}=f_1.
\]
Then by assumption of Theorem \ref{thm_main}, there exists an open neighborhood $\tilde F$ of $F(x_0)=\p \Omega_-$ in $\p \Omega$, such that 
\[
\p_\nu u_1|_{\tilde F}=\p_\nu u_2|_{\tilde F},\quad \p_\nu (\Delta u_1)|_{\tilde F}=\p_\nu (\Delta u_2)|_{\tilde F}. 
\]
We have
\begin{equation}
\label{eq_need_bce}
\mathcal{L}_{A^{(1)},q^{(1)}}(u_1-u_2)=(A^{(2)}-A^{(1)})\cdot Du_2 + (q^{(2)}-q^{(1)})u_2\quad \textrm{in}\quad  \Omega. 
\end{equation}
Let $v\in H^{4}(\Omega)$ satisfy $ 
\mathcal{L}_{A^{(1)},q^{(1)}}^* v=0$ in $\Omega$. 
Using \eqref{eq_Green_form}, we get
\begin{equation}
\label{eq_identity_main}
\begin{aligned}
\int_{\Omega} ((A^{(2)}-A^{(1)})\cdot Du_2)\overline{v}dx + \int_{\Omega} (q^{(2)}-q^{(1)})u_2 \overline{v}dx\\
=-\int_{\p \Omega\setminus{\tilde F}} \p_{\nu} (-\Delta(u_1-u_2))\overline{v}dS - \int_{\p \Omega\setminus{\tilde F}} \p_{\nu}(u_1-u_2)\overline{(-\Delta v)}dS. 
\end{aligned}
\end{equation}

To show the equalities $A^{(1)}=A^{(2)}$ and $q^{(1)}=q^{(2)}$, the idea is to use the identity \eqref{eq_identity_main} with $u_2$ and $v$ being complex geometric optics solutions and to use the boundary Carleman estimate \eqref{eq_boundary_Carleman} to show that the boundary integrals  in \eqref{eq_identity_main},  multiplied by some power of $h$,  tend to zero as $h\to 0$. 

To construct the appropriate complex geometric optics solutions, let $\varphi$ and $\psi$ be defined by \eqref{eq_varphi} and \eqref{eq_psi}, respectively. Then thanks to  Proposition \ref{prop_complex_geom_optics}, we can take
\begin{equation}
\label{eq_u_2_optics}
u_2(x;h)=e^{\frac{\varphi+i\psi}{h}}(a_0^{(2)}(x)+ha_1^{(2)}(x)+r^{(2)}(x;h)),
\end{equation}
\begin{equation}
\label{eq_v_optics}
v(x;h)=e^{\frac{-\varphi+i\psi}{h}}(a_0^{(1)}(x)+ha_1^{(1)}(x)+r^{(1)}(x;h)),
\end{equation}
where $a_0^{(j)}\in C^\infty(\overline{\Omega})$ and $a_1^{(j)}\in C^4(\overline{\Omega})$, $j=1,2$, are  such that 
\begin{equation}
\label{eq_transport_a_1,a_2}
\begin{aligned}
\bigg( (\nabla \varphi + i \nabla \psi)\cdot \nabla + \frac{\Delta\varphi + i \Delta \psi}{2}\bigg)^2 a_0^{(2)}=0,\\
\bigg( (-\nabla \varphi + i \nabla \psi)\cdot \nabla + \frac{-\Delta\varphi + i \Delta \psi}{2}\bigg)^2 a_0^{(1)}=0,
\end{aligned}
\end{equation}
and
\begin{equation}
\label{eq_remainder_b}
\|r^{(j)}\|_{H^4_{\textrm{scl}}}=\mathcal{O}(h^2).
\end{equation}

Substituting $u_2$ and $v$, given by \eqref{eq_u_2_optics} and \eqref{eq_v_optics}, in \eqref{eq_identity_main}, we get
\begin{equation}
\label{eq_identity_main_2}
\begin{aligned}
\int_\Omega (A^{(2)}-A^{(1)})\cdot\frac{1}{h}(D\varphi + i D\psi)(a_0^{(2)}+ha_1^{(2)}+r^{(2)})(\overline{a_0^{(1)}}+h\overline{a_1^{(1)}}+\overline{r^{(1)}})dx \\
+\int_{\Omega} (A^{(2)}-A^{(1)})\cdot (D a_0^{(2)}+h D a_1^{(2)}+ Dr^{(2)})(\overline{a_0^{(1)}}+h\overline{a_1^{(1)}}+\overline{r^{(1)}})dx\\
+\int_\Omega (q^{(2)}-q^{(1)})(a_0^{(2)}+ha_1^{(2)}+r^{(2)})(\overline{a_0^{(1)}}+h\overline{a_1^{(1)}}+\overline{r^{(1)}})dx\\
=-\int_{\p \Omega\setminus{\tilde F}} \p_{\nu} (-\Delta(u_1-u_2))\overline{v}dS - \int_{\p \Omega\setminus{\tilde F}} \p_{\nu}(u_1-u_2)\overline{(-\Delta v)}dS. 
\end{aligned}
\end{equation}

Let us now show that 
\begin{equation}
\label{eq_boundary_int_1}
h \int_{\p \Omega\setminus{\tilde F}} \p_{\nu} (-\Delta(u_1-u_2))\overline{v}dS\to 0,\quad \textrm{as}\quad h\to +0,
\end{equation}
and 
\begin{equation}
\label{eq_boundary_int_2}
h\int_{\p \Omega\setminus{\tilde F}} \p_{\nu}(u_1-u_2)\overline{(-\Delta v)}dS\to 0,\quad \textrm{as}\quad h\to +0,
\end{equation}
where $u_2$ and $v$, given by \eqref{eq_u_2_optics} and \eqref{eq_v_optics}.  
To this end, notice that \eqref{eq_remainder_b} implies that 
$r^{(j)}=\mathcal{O}(1)$, $j=1,2$, in the standard  Ê  ($h=1$)  $H^2(\Omega)$--norm. Hence, 
\begin{equation}
\label{eq_remainder_b_1} 
r^{(j)}|_{\p \Omega}=\mathcal{O}(1) \quad \textrm{in}\quad L^2(\p \Omega),
\end{equation}
and 
\begin{equation}
\label{eq_remainder_b_2} 
\nabla r^{(j)}|_{\p \Omega}=\mathcal{O}(1) \quad \textrm{in}\quad L^2(\p \Omega). 
\end{equation}
Moreover, it follows from  \eqref{eq_remainder_b} that $r^{(j)}=\mathcal{O}(1/h)$, $j=1,2$, in the standard  Ê  $H^3(\Omega)$--norm. Thus, 
\begin{equation}
\label{eq_remainder_b_3} 
\Delta r^{(j)}|_{\p \Omega}=\mathcal{O}(1/h) \quad \textrm{in}\quad L^2(\p \Omega). 
\end{equation}
Furthermore, by the definition of $F(x_0)$ and $\tilde F$,  there exists $\varepsilon >0$ fixed  such that 
\[
\p\Omega_-=F(x_0)\subset F_\varepsilon: =\{x\in \p \Omega:\p_\nu\varphi\le \varepsilon\}\subset \tilde F. 
\]

Using the Cauchy--Schwarz inequality and \eqref{eq_remainder_b_1},
 we have
\begin{align*}
&\bigg| h \int_{\p \Omega\setminus{\tilde F}} \p_{\nu} (-\Delta(u_1-u_2))\overline{v}dS \bigg|\\
 &\le 
h\int_{\p \Omega\setminus{F_\varepsilon}} |\p_{\nu} (-\Delta(u_1-u_2))| e^{\frac{-\varphi}{h}}|a_0^{(1)}+ha_1^{(1)}+r^{(1)}|dS\\
&\le \mathcal{O}(h)\bigg(\int  _{\p \Omega\setminus{F_\varepsilon}} \varepsilon |\p_{\nu} (-\Delta(u_1-u_2))|^2 e^{\frac{-2\varphi}{h}}dS\bigg)^{1/2}\|a_0^{(1)}+ha_1^{(1)}+r^{(1)}\|_{L^2(\p \Omega)}\\
&\le \mathcal{O}(h)\|\sqrt{\p_{\nu}\varphi}e^{\frac{-\varphi}{h}}\p_{\nu} (-\Delta(u_1-u_2))\|_{L^2(\p \Omega_+)}.
\end{align*}
By the boundary Carleman estimate \eqref{eq_boundary_Carleman} and \eqref{eq_need_bce}, we get
\begin{equation}
\label{eq_similar_bound}
\begin{aligned}
&\mathcal{O}(h)\|\sqrt{\p_{\nu}\varphi}e^{\frac{-\varphi}{h}}\p_{\nu} (-\Delta(u_1-u_2))\|_{L^2(\p \Omega_+)} \le \mathcal{O}(h^{3/2})\|e^{\frac{-\varphi}{h}}\mathcal{L}_{A^{(1)},q^{(1)}}(u_1-u_2)\|_{L^2(\Omega)}\\
&=\mathcal{O}(h^{3/2}) \| e^{\frac{-\varphi}{h}}(A^{(2)}-A^{(1)})\cdot Du_2+e^{\frac{-\varphi}{h}}(q^{(2)}-q^{(1)})u_2 \|_{L^2(\Omega)}\\
&\le \mathcal{O}(h^{3/2}) \| (A^{(2)}-A^{(1)})\cdot \frac{1}{h}(D\varphi+iD\psi)(a_0^{(2)}+ha_1^{(2)}+r^{(2)})\|_{L^2(\Omega)}\\ 
&+ \mathcal{O}(h^{3/2}) \| (A^{(2)}-A^{(1)})\cdot (Da_0^{(2)}+hDa_1^{(2)}+Dr^{(2)})\|_{L^2(\Omega)}\\
&+ \mathcal{O}(h^{3/2}) \| (q^{(2)}-q^{(1)})(a_0^{(2)}+ha_1^{(2)}+r^{(2)})\|_{L^2(\Omega)}
\le \mathcal{O}(h^{1/2}).
\end{aligned}
\end{equation}
Thus, \eqref{eq_boundary_int_1} follows.

To establish \eqref{eq_boundary_int_2}, we first notice that thanks to \eqref{eq_eikonal}, we have
\begin{align*}
\Delta v=&e^{\frac{-\varphi+i\psi}{h}}\bigg(\frac{-\Delta\varphi+i\Delta\psi}{h}(a_0^{(1)}+ha_1^{(1)}+r^{(1)})\\
&+2\frac{-\nabla\varphi+i\nabla\psi}{h}\cdot (\nabla a_0^{(1)}+h\nabla a_1^{(1)}+\nabla r^{(1)})+ \Delta a_0^{(1)}+h\Delta a_1^{(1)}+\Delta r^{(1)}.
\bigg)
\end{align*}
This together with \eqref{eq_remainder_b_1}, \eqref{eq_remainder_b_2} and \eqref{eq_remainder_b_3} imply that
\[
\Delta v|_{\p \Omega}=e^{\frac{-\varphi}{h}} \tilde v, \quad \tilde v=\mathcal{O}(1/h)\quad \textrm{in}\quad L^2(\p \Omega). 
\]  
This and the Cauchy-Schwarz inequality yield that 
\begin{align*}
\bigg|h\int_{\p \Omega\setminus{\tilde F}} \p_{\nu}(u_1-u_2)\overline{(-\Delta v)}dS\bigg|\le \mathcal{O}(1)\|\sqrt{\p_\nu \varphi}e^{\frac{-\varphi}{h}}\p_\nu(u_1-u_2)\|_{L^2(\p \Omega_+)}\\
\le \mathcal{O}(h^{3/2})\|e^{\frac{-\varphi}{h}}\mathcal{L}_{A^{(1)},q^{(1)}}(u_1-u_2)\|_{L^2(\Omega)}\le \mathcal{O}(h^{1/2}).\\
\end{align*}
Here we use the boundary Carleman estimate \eqref{eq_boundary_Carleman} and proceed similarly to \eqref{eq_similar_bound}.  Hence, \eqref{eq_boundary_int_2} follows.

\section{Determining the first order perturbation}

Multiplying \eqref{eq_identity_main_2} by $h$ and letting $h\to+0$,  and using \eqref{eq_boundary_int_1} and \eqref{eq_boundary_int_2}, we get
\begin{equation}
\label{eq_identity_main_3}
\int_\Omega(A^{(2)}-A^{(1)})\cdot (\nabla \varphi + i \nabla \psi)a_0^{(2)}\overline{a_0^{(1)}}dx=0,
\end{equation}
where $a_0^{(1)},a_0^{(2)}\in C^\infty(\overline{\Omega})$ satisfy the transport equations \eqref{eq_transport_a_1,a_2}.  

Consider now \eqref{eq_identity_main_3} with $a_0^{(1)}=e^{\Phi_1}$ and $a_0^{(2)}=e^{\Phi_2}$ such that 
\begin{equation}
\label{eq_transport_recovery}
\begin{aligned}
 (\nabla \varphi + i \nabla \psi)\cdot \nabla \Phi_2 + \frac{\Delta\varphi + i \Delta \psi}{2}=0,\\
(\nabla \varphi + i \nabla \psi)\cdot \nabla \overline{\Phi_1} + \frac{\Delta\varphi + i \Delta \psi}{2}=0,
\end{aligned}
\end{equation}
and $\Phi_j\in C^\infty(\overline{\Omega})$, $j=1,2$. 
In the coordinates $(z,\theta)$, introduced in Section \ref{sec_CGO}, the equations \eqref{eq_transport_recovery} have the following form,
\begin{align*}
\p_{\bar z} \Phi_2-\frac{(n-2)}{2(z-\bar z)}=0,\quad \p_{\bar z} \overline{\Phi_1}-\frac{(n-2)}{2(z-\bar z)}=0.
\end{align*}
Hence,
\begin{equation}
\label{eq_miracle}
\p_{\bar z} (\Phi_2+\overline{\Phi_1})-\frac{n-2}{z-\bar z}=0.
\end{equation}
Notice that $ge^{\Phi_2}$ with $g\in C^\infty(\overline{\Omega})$ such that 
\begin{equation}
\label{eq_g}
(\nabla \varphi + i \nabla \psi)\cdot \nabla g=0\quad \textrm{in}\quad \Omega,
\end{equation}
also satisfies the transport equation \eqref{eq_transport_a_1,a_2}. 
In the coordinates $(z,\theta)$, the condition \eqref{eq_g} reads $\p_{\bar z} g=0$ in $\Omega$. 

Substituting  $a_0^{(2)}=ge^{\Phi_2}$ and $a_0^{(1)}=e^{\Phi_1}$  in \eqref{eq_identity_main_3}, we get
\begin{equation}
\label{eq_identity_main_4}
\int_\Omega(A^{(2)}-A^{(1)})\cdot (\nabla \varphi + i \nabla \psi)ge^{\Phi_2+\overline{\Phi_1}}dx=0.
\end{equation}
In the cylindrical coordinates $(z,\theta)$, by \eqref{eq_nabla_cyl}, the identity \eqref{eq_identity_main_4} has the form,
\[
\int_\Omega(A^{(2)}-A^{(1)})\cdot (e_1+i e_r)\frac{g}{z}e^{\Phi_2+\overline{\Phi_1}}r^{n-2}drd\theta dx_1=0.
\]
Let $P_\theta$ be the two-dimensional plane consisting of points $(x_1,r\theta)$ for $\theta$ fixed, and write $\Omega_\theta=\Omega\cap P_\theta$.  We also use the complex variable $z=x_1+ir$, which identifies $P_\theta$ with $\C$. 

Taking $g=g_1(z)\otimes g_2(\theta)\in C^\infty(\overline{\Omega})$, where $g_1$ is holomorphic and varying $g_2$, we obtain as in \cite{DKSU_2007}, for almost all $\theta\in \mathbb{S}^{n-2}$,
\[
\int_{\Omega_\theta}(A^{(2)}-A^{(1)})\cdot (e_1+i e_r)\frac{g_1}{z} e^{\Phi_2+\overline{\Phi_1}}(z-\bar z)^ {n-2}dz\wedge d\bar z=0,
\]
and hence, for all $\theta \in \mathbb{S}^{n-2}$, by continuity. 
Since $\Im z>0$ in $\Omega$, the function $g_1/z$ is an arbitrary holomorphic function and therefore, we can drop the factor $1/z$. 
We get
\[
\int_{\Omega_\theta}(A^{(2)}-A^{(1)})\cdot (e_1+i e_r) g_1 e^{\Phi_2+\overline{\Phi_1}}(z-\bar z)^ {n-2}dz\wedge d\bar z=0.
\]
By \eqref{eq_miracle}, we conclude that
\begin{equation}
\label{eq_miracle_2}
\p_{\bar z}(e^{\Phi_2+\overline{\Phi_1}}(z-\bar z)^ {n-2})=0.
\end{equation}
Since $\Im z>0$ in $\Omega$, the holomorphic function $e^{\Phi_2+\overline{\Phi_1}}(z-\bar z)^ {n-2}$ is nowhere vanishing, and we can choose $g_1=(e^{\Phi_2+\overline{\Phi_1}}(z-\bar z)^ {n-2})^{-1}g_0$, where $\p_{\bar z}g_0=0$ in $\Omega$. We get 
\begin{equation}
\label{eq_g_0}
\int_{\Omega_\theta}(A^{(2)}-A^{(1)})\cdot (e_1+i e_r)g_0dz\wedge d\bar z=0.
\end{equation}

Choosing complex geometric optics solutions $u_2$ and $v$ as in \eqref{eq_u_2_optics} and \eqref{eq_v_optics}, where $\psi$ is replaced by $-\psi$, we conclude, by repeating the arguments above, that also
\begin{equation}
\label{eq_tilde_g_0}
\int_{\Omega_\theta}(A^{(2)}-A^{(1)})\cdot (e_1-i e_r)\tilde g_0 dz\wedge d\bar z=0,
\end{equation}
where  $\p_{z} \tilde g_0=0$ in $\Omega$.

Thus, since any  $\xi\in P_\theta$ is a linear combination of $e_1$ and $e_r$, choosing $g_0=\tilde g_0=1$ in 
\eqref{eq_g_0} and \eqref{eq_tilde_g_0}, 
 we get 
\[
\int_{\Omega_\theta}(A^{(2)}-A^{(1)})\cdot \xi dz\wedge d\bar z=0 \quad\textrm{for all } \xi\in P_\theta. 
\]
At this point we are exactly in the same situation as the one, described in \cite[Section 5]{DKSU_2007}, see formula (5.7) there. Repeating the arguments given in that paper, following this formula,  we obtain that $d(A^{(2)}-A^{(1)})=0$ in $\Omega$.  
Here $A^{(1)}$ and $A^{(2)}$ are viewed as $1$--forms.   We may notice  that the arguments of \cite[Section 5]{DKSU_2007} are based upon the microlocal
Helgason support theorem combined with the microlocal Holmgren theorem.

Since $\Omega$ is simply connected, we may write 
\begin{equation}
\label{eq_first_A_2-A_1}
A^{(2)}-A^{(1)}=\nabla \Psi,
\end{equation}
 with $\Psi\in C^{5}(\overline{\Omega})$. 

Our next step is to prove that $\Psi=0$ on $\overline{\Omega}$. Writing \eqref{eq_g_0} and \eqref{eq_tilde_g_0} in the Euclidian coordinates and using \eqref{eq_first_A_2-A_1} and \eqref{eq_nabla_cyl}, we have
\begin{align*}
\int_{\Omega_\theta} (x_1+ir )(\nabla \varphi+ i\nabla \psi) \cdot (\nabla \Psi) g_0 dx_1 d r=0,\\ 
\int_{\Omega_\theta} (x_1-ir )(\nabla \varphi- i\nabla \psi) \cdot (\nabla \Psi) \tilde g_0 dx_1 d r=0.
\end{align*}
Thus, by  \eqref{eq_nabla_cyl_2} we get
\[
\int_{\Omega_\theta} (\p_{\bar z}\Psi) g_0 dz\wedge d \bar z=0,\quad 
\int_{\Omega_\theta} (\p_{ z}\Psi) \tilde g_0 dz\wedge d \bar z=0,
\]
for all $\theta \in \mathbb{S}^{n-2}$,  $g_0, \tilde g_0\in C^\infty(\overline{\Omega_\theta}) $ such that $\p_{\bar z}  g_0=0$ in $\Omega_\theta$,  and $\p_{z} \tilde g_0=0$.
By Sard's theorem, the boundary of $\Omega_\theta$ is smooth for almost all $\theta$, see  \cite{DKSU_2007, KnuSalo_2007} for more details. 
Thus, by Stokes' theorem,  we get 
\[
\int_{\p \Omega_\theta} \Psi g_0 dz=0,\quad 
\int_{\p \Omega_\theta} \Psi \tilde g_0  d \bar z=0,
\]
for almost all $\theta\in \mathbb{S}^{n-2}$. 
Hence, taking $\tilde g_0=\overline{g_0}$, we obtain that
\[
\int_{\p \Omega_\theta} \overline{\Psi}g_0dz=0,
\]
and therefore,
\[
\int_{\p \Omega_\theta} (\Re \Psi) g_0dz=0,\quad \int_{\p \Omega_\theta} (\Im \Psi) g_0dz=0,
\]
for any $g_0\in C^\infty(\overline{\Omega_\theta}) $ such that $\p_{\bar z}  g_0=0$ and for almost all $\theta$.  At this point we are precisely in the same situation as the one, described in the beginning of  \cite[Section 6]{DKSU_2007}. Repeating the arguments of that paper, we conclude that $\Psi$ is constant along the connected set $\p \Omega$, and we may and shall assume that $\Psi=0$ along $\p \Omega$.

Going back to \eqref{eq_identity_main_3}, we get
\begin{equation}
\label{eq_identity_main_final_A}
\int_\Omega(\nabla\Psi)\cdot (\nabla \varphi + i \nabla \psi)a_0^{(2)}\overline{a_0^{(1)}}dx=0,
\end{equation}
where $a_0^{(1)},a_0^{(2)}\in C^\infty(\overline{\Omega})$ satisfy
the transport equations \eqref{eq_transport_a_1,a_2}.  
Integrating by parts in \eqref{eq_identity_main_final_A} and using the fact that $\Psi=0$ along $\p \Omega$, we obtain that
\[
\int_{\Omega} \Psi\bigg((\Delta\varphi +i\Delta \psi) +(\nabla \varphi + i \nabla \psi )\cdot\nabla \bigg) a_0^{(2)}\overline{a_0^{(1)}}dx=0.
\]
This implies that 
\begin{equation}
\label{eq_identity_main_final_A_2}
\int_{\Omega} \Psi\bigg((Ta_0^{(2)})\overline{a_0^{(1)}}+ a_0^{(2)} T\overline{a_0^{(1)}}\bigg)dx=0,
\end{equation}
where the operator $T$ is given by  \eqref{eq_operator_T}.
We choose now $a_0^{(1)}\in C^\infty(\overline{\Omega})$ being a solution of 
the equation $T\overline{a_0^{(1)}}=0$ of the form $a_0^{(1)}=e^{\Phi_1}$. As for 
$a_0^{(2)}\in C^\infty(\overline{\Omega})$, we require that  $Ta_0^{(2)}=e^{\Phi_2}g$,
where $\Phi_2\in C^\infty(\overline{\Omega})$ is such that  $Te^{\Phi_2}=0$ and $g\in C^\infty(\overline{\Omega})$ is such that 
$(\nabla \varphi + i \nabla \psi )\cdot\nabla g=0$.  It is clear that $T^2a_0^{(2)}=0$. The existence of such $a_0^{(1)}$ and $a_0^{(2)}$ is explained in Section \ref{sec_CGO}. Thus, it follows from
\eqref{eq_identity_main_final_A_2} that
\begin{equation}
\label{eq_identity_main_final_A_3}
\int_{\Omega} \Psi e^{\Phi_2+\overline{\Phi_1}}gdx=0. 
\end{equation}
In the coordinates $(z,\theta)$, \eqref{eq_identity_main_final_A_3} reads
\begin{equation}
\label{eq_identity_main_final_A_4}
\int_{\Omega} \Psi g e^{\Phi_2+\overline{\Phi_1}}(z-\bar z)^{n-2}dz\wedge d\bar z\wedge d\theta=0. 
\end{equation}
As before,  the function $e^{\Phi_2+\overline{\Phi_1}}(z-\bar z)^ {n-2}$ is nowhere vanishing holomorphic in $z$, see \eqref{eq_miracle_2}, and we shall take 
 $g=(e^{\Phi_2+\overline{\Phi_1}}(z-\bar z)^ {n-2})^{-1}\otimes g_2(\theta)$, where $g_2$ is smooth. Hence, 
\eqref{eq_identity_main_final_A_4} implies that
\[
\int_{\Omega} \Psi(x_1,r,\theta) g_2(\theta) dx_1drd\theta=0,
\]
for any smooth function $g_2(\theta)$. 
We arrive exactly at the formula (6.2) of the paper \cite{DKSU_2007}, and arguing as in \cite{DKSU_2007}, appealing, as before,  to results of analytic microlocal analysis, we obtain that $\Psi=0$ in $\Omega$. 
Hence, we conclude that $A^{(1)}=A^{(2)}$ in $\Omega$.

The final  step in proving Theorem \ref{thm_main} is to show that $q^{(1)}=q^{(2)}$ in $\Omega$.  To this end, substituting $A^{(1)}=A^{(2)}$ in \eqref{eq_identity_main_2} and letting $h\to+0$, we get
\begin{equation}
\label{eq_identity_main_q}
\int_{\Omega}(q^{(2)}-q^{(1)})a_0^{(2)}\overline{a_0^{(1)}}dx=0.
\end{equation}
Here we use the fact that 
\begin{equation}
\label{eq_boundary_int_for_q}
\bigg |\int_{\p \Omega\setminus{\tilde F}} \p_{\nu} (-\Delta(u_1-u_2))\overline{v}dS\bigg| \le \mathcal{O}(h^{1/2}), 
\bigg| \int_{\p \Omega\setminus{\tilde F}} \p_{\nu}(u_1-u_2)\overline{(-\Delta v)}dS\bigg| \le \mathcal{O}(h^{1/2}),
\end{equation}
where $u_2$ and $v$, given by \eqref{eq_u_2_optics} and \eqref{eq_v_optics}.  Notice that \eqref{eq_boundary_int_for_q} is obtained similarly to 
\eqref{eq_boundary_int_1} and \eqref{eq_boundary_int_2} but under the assumption that $A^{(1)}=A^{(2)}$. 

We choose $a_0^{(1)},a_0^{(2)}\in C^\infty(\overline{\Omega})$ to be such that $a_0^{(1)}=e^{\Phi_1}$  and $a_0^{(2)}=e^{\Phi_2}g$,
where $\Phi_1, \Phi_2\in C^\infty(\overline{\Omega})$ are such that $Te^{\overline{\Phi_1}}=0$, $Te^{\Phi_2}=0$, and $g\in C^\infty(\overline{\Omega})$ is such that 
$(\nabla \varphi + i \nabla \psi )\cdot\nabla g=0$.  Thus, 
\eqref{eq_identity_main_q} yields that 
\[
\int_{\Omega}(q^{(2)}-q^{(1)})  e^{\Phi_2+\overline{\Phi_1}}gdx=0.
\]
Arguing in the same way as after \eqref{eq_identity_main_final_A_3}, we conclude that $q^{(1)}=q^{(2)}$. This completes the proof of Theorem \ref{thm_main}.

\section*{Acknowledgements}  

The research of K.K. is supported by the
Academy of Finland (project 125599).  The research of M.L. 
is partially supported 
 by the Academy of Finland Center of Excellence programme 213476. The research of
G.U. is partially supported by the National Science Foundation.

\end{document}